\documentclass[11pt,leqno, twoside]{amsart}
\usepackage[foot]{amsaddr}
\usepackage{amssymb,amsfonts}
\usepackage{amsmath,amsthm,amsxtra}
\usepackage[]{geometry}
\usepackage{epsfig}
\usepackage{slashed}

\usepackage{color}
\usepackage{verbatim}

\newcommand{\Aut}{{\rm Aut}}

\newcommand{\End}{\mathop{\rm End }}

\newcommand{\C}{\mathcal{C}}

\newcommand{\V}{\mathcal{V}}

\newcommand{\FF}{\mathbb{F}}

\newcommand{\ZZ}{\mathbb{Z}}

\newcommand{\fg}{\mathfrak{g}}

\newcommand\bl{(\, . \, | \, . \, )}

\newtheorem{theorem}{Theorem}

\newtheorem{lemma}{Lemma}
\newtheorem{corollary}{Corollary}

\newtheorem{conjecture}{Conjecture}

\newtheorem{remark}{Remark}
\newtheorem{example}{Example}

\renewcommand{\epsilon}{\varepsilon}

\newcommand{\half}{\frac{1}{2}}

\newcommand{\la}{\lambda}

\newcommand{\al}{\alpha}

\newcommand{\epi}{\End \pi}

\begin{document}
	
\title{On Complexity of Representations of Quivers}


\author{Victor G.~Kac$ ^1 $} 
\address{$ ^1 $Department of Mathematics, M.I.T, 
	Cambridge, MA 02139, USA. Email:  kac@math.mit.edu Supported in part by the Bert and Ann Kostant fund.}
\maketitle

	\thispagestyle{empty}

\section*{Abstract} 
It is shown that, given a representation of a quiver over a finite field, one can check in polynomial time whether it is absolutely indecomposable.

\section{Some results on absolutely indecomposable representations of quivers}
Let $ \Gamma $ be a finite graph without self-loops (but several edges connecting two vertices are allowed), and let  $ \mathcal{V} $ denote the set of its vertices.
The graph $ \Gamma $ with an orientation $\Omega$ of its edges is called a \emph{quiver}.  A \textit{representation} of the quiver $(\Gamma, \Omega)$ over a field $ \mathbb{F} $ is a collection of finite-dimensional vector spaces $\{ U_v  \}_{v \in \mathcal{V}} $ over $\mathbb{F}$ and linear maps
$ \{ U_{v} \rightarrow {U}_{w} \} $ for each oriented edge $ v \rightarrow w $.
Homomorphisms and isomorphisms of two representations are defined in the obvious way.  The \emph{direct sum} of two representations $ ( \{{U}_v\}, \{{U}_{v} \rightarrow {U}_{w} \}  ) $
and $ ( \{{U}'_v\}, \{{U}'_{v} \rightarrow {U'}_{w} \}  )$  is the representation
\[ (  \{{U}_v  \oplus {U'}_v        \}, \{{U}_{v}  \oplus  {U'}_{v}  \rightarrow
   {U}_{w}  \oplus  {U'}_{w}        \}     ), \]
   where maps are the direct sums of maps.
A representation $\pi$ is called \textit{indecomposable} if it is not isomorphic to a direct sum of two non-zero representations;
$\pi$ is called \emph{absolutely indecomposable} if it is indecomposable over the algebraic closure $\bar{\FF}$ of the field $ \FF $.

Let $ r = \# \V $ and let $ Q = \underset{v \in \V}{\bigoplus} \ \ZZ \alpha_v  $
be a free abelian group of rank $ r $ with a fixed basis 
$ \{\alpha_v\}_{v \in \mathcal{V}} $.  Let $ Q_+ = \underset{v}{\bigoplus} \ZZ_{\geq 0} \, \alpha_v \subset Q $. 
The \emph{dimension} of a representation $ \pi  = \{{U}_v \} _{v\in \mathcal{V}}    $
is the element 
\[ \dim \pi = \sum_{v \in \mathcal{V} } (\dim \, {U}_v    ) \alpha_v  \in Q_+. \]
The \emph{Cartan matrix}  of the graph $ \Gamma $  is the symmetric matrix $ A = ( a_{uv}    )_{u,v \in \mathcal{V}} $, where \ $ a_{vv} = 2     $  and $  -a_{uv} $ is the number of edges, connecting $ u $ and $ v $ if $ u \neq v $.  Define a
$ \half \ZZ $-valued symmetric bilinear form
on $ Q $, such that $(\alpha|\alpha)\in \ZZ$, by
\[( \alpha_u  | \alpha_v )  =  \half a_{uv},\ u,v\in \mathcal{V},\] 
and the following (involutive) automorphisms $ r_v,\, v \in \mathcal{V}$, of the free abelian group $ Q $ 
\[r_v (\alpha_u) = \alpha_u - a_{uv}\alpha_v,  \ u \in \mathcal{V}.\] 
The group $W \subset \Aut \, Q   $, generated by all $ r_v, v \in \mathcal{V}  $, is called the \emph{Weyl group} of the graph $\Gamma$. It is immediate to see that the bilinear form $ \bl $ is invariant with respect to all $ r_v, v \in \mathcal{V} $, hence with respect to the Weyl group $ W $.

It is well known that the group $ W $ is finite if and only if the Cartan matrix $A$ is posiive definite, which happens if and only if all connected components of $ \Gamma $ are Dynkin diagrams of simple finite-dimensional Lie algebra of type $A_r, D_r,
E_6, E_7, E_8$ (see e.g.   \cite{K90} ). Gabriel's theorem \cite{G72} states that for a quiver $ (\Gamma, \Omega) $ the number of indecomposable representations, up to isomorphism, is finite if and only if the group $ W $ is finite. Moreover, in this case the map $ \pi \mapsto \dim \, \pi $ establishes a bijective correspondence between isomorphism classes of indecomposable representations of $(\Gamma, \Omega)$ and the
set of positive roots $ \Delta_+ \subset Q_+ $ of the semisimple Lie algebra with Dynkin diagram $ \Gamma $, where 
\begin{equation}\label{eq3}
 \Delta_+ = \underset{v \in \mathcal{V}}{\bigcup} \left( (W \cdot \alpha_v      )  \cap Q_+     \right). 
\end{equation}  

For an arbitrary graph $ \Gamma $ denote by $ \Delta^{\text{re}}_+ $ the RHS of \eqref{eq3}; note that $ (\alpha | \alpha    ) = 1 $ for all $ \alpha \in \Delta^{\text{re}}_+ $.
Furthermore, let
\begin{equation}\label{eq4}
  \mathcal{C} = \{\alpha \in Q_+ \backslash \{0 \} \, | \,  (\alpha | \alpha_v) \leq 0, v\in \mathcal{V}\,, \hbox{and}\, \operatorname{supp} \alpha \, \hbox{is connected}\},
\end{equation}  
where  for $ \alpha = \sum_{v \in \mathcal{V}} n_v \alpha_v $, we let $ \operatorname{supp} \alpha = \{v |\,\, n_v \neq 0 \} $.
We let 
\[ \Delta^\text{im}_+ = W \cdot  \mathcal{C} , \quad \Delta_+ = \Delta^\text{re}_+ \cup \Delta^\text{im}_+ .\]
It is easy to see that $ \Delta^\text{im}_+ \subset Q_+ $ and that $ (\alpha | \alpha)  \in \ZZ_{\leq 0}$ for $ \alpha \in \Delta^\text{im}_+ $. The set $ \Delta_+ \subset Q_+ $ is the set of \emph{positive roots} of the Kac-Moody algebra $ \mathfrak{g} (A)  $, associated to the Cartan matrix $ A  $, and
$ \Delta^\text{im}_+ $ is empty if and only if
the matrix $A$ is positive definite \cite{K80}, \cite{K90}.

\begin{theorem}\label{th1}
Let $ \FF = \FF_q $ be a field of $ q $ elements.
\begin{enumerate}
\item [(a)]  The number of absolutely indecomposable representations over $ \FF_q $ of dimension $ \alpha \in Q_+ $ of a quiver $ (\Gamma, \Omega) $ is independent of the orientation
  $\Omega$  . It is zero if $ \alpha \notin \Delta_+ $, and it is given by a monic
  polynomial $ P_{\Gamma , \alpha}  (q) $ of degree $ 1 - (\alpha | \alpha) $ with integer
  coefficients. In particular, $P_{\Gamma, \alpha}(q)=1$ if $\alpha \in \Delta^\text{re}_+$. 
\item [(b)]  The constant term $ P_{\Gamma , \alpha}  (0) $ equals to the multiplicity of the root $ \alpha  $ in $ \fg(A) $.
\item [(c)]  All coefficients of $ P_{\Gamma, \alpha }(q) $ are non-negative.
\item [(d)]  Consequently, for any quiver $(\Gamma, \Omega)$ and any $\alpha\in \Delta_+$
  there exists an absolutely indecomposable representation over $\FF_q$ of dimension
  $\alpha$.
\end{enumerate}
\end{theorem}

Claim (a) was proved in \cite{K80} and \cite{K83}; claims (b) and (c) were conjectured in
\cite{K80}, \cite{K83},
and proved in \cite{H10} and \cite{HLRV13} respectively. For indivisible
$\alpha\in \Delta_+$ both claims (b) and (c) were proved earlier in \cite{CBVB04}.

\section{Quasi-nilpotent subalgebras of $\End_\FF U$.}
Consider a finite-dimensional vector space $ U $ over a field $ \FF $. An endomorphism $ a $ of $ U $ is called \emph{quasi-nilpotent} if all its eigenvalues are equal; denote these eigenvalues by $ \operatorname{eig}(a) $. They are  elements of the algebraic closure $\bar{\FF}$
of the field $ \FF $. An associative subalgebra $A$ of $\End_\FF U$ is called \emph{quasi-nilpotent} if it consists of quasi-nilpotent elements.
For an associative algebra $ A $ we denote by $ A_- $ the Lie algebra obtained from $ A $ by taking the bracket $ [a,b] = ab-ba   $. We also let $\overline{A} =
\overline{\FF} \otimes_\FF A $, $\overline {U}=\overline{\FF} \otimes_\FF U $. 
\begin{lemma}\label{Lem1}
	Let $ A $ be a subalgebra of the associative algebra $ \End_\FF U $.
\begin{enumerate}
\item [(a)] If A is a quasi-nilpotent subalgebra, then in some basis of $ \overline{U}$,
  all endomorphisms $ a \in $ A have upper triangular matrices with $ eig(a) $ on the diagonal. In particular, $\operatorname{eig}(a+b) = \operatorname{eig}(a) + \operatorname{eig}(b)$
   for $a,b \in A$, and
$ A_- $ is a nilpotent Lie algebra.
	\item [(b)] If $ A_- $ is a nilpotent Lie algebra and A has a basis, consisting of quasi-nilpotent endomorphisms, then  $ A $ is a  quasi-nilpotent subalgebra.
\end{enumerate}
\end{lemma} 
\begin{proof}
  Burnside's theorem says that any subalgebra of the
  $ \overline{\FF} $
  -algebra   $\End_ {\overline{\FF}} \overline{U} $
  , where $ \overline{U} $ is a finite-dimensional vector space over $ \overline{\FF} $, which acts irreducibly on $ \overline{U} $, coincides with $\End \overline{U} $. Hence, in some basis of $\overline{U}$ the algebra $ \overline{A} $ consists of upper triangular block matrices with blocks $\End_{ \overline{\FF}}  \overline{\FF}^{m_i} $ on the diagonal, where $ m_i \geq 1 , \sum_i m_i = \dim \bar{U} $.  

If $ A $ is a quasi-nilpotent subalgebra, then so is $ \overline{A} $, and,
in particular, $ \End_ {\overline{\FF}} \overline{\FF}^{m_i} $ for all $ i $. This implies that all $ m_i = 1 $. Hence $ \overline{A} $ consists of upper triangular quasi-nilpotent matrices. This proves (a).

In order to prove $ (b) $, note that if $ A_- $ is a nilpotent Lie algebra, then so is
$\bar{A}_-$, and, in particular so are all $ (\End_{\bar{\FF}} \overline{\FF}^{m_i})_- $. It follows that all $ m_i = 1 $, so that $ \overline{A}_- $ consists of upper triangular matrices in some basis of $ \overline{U} $. Since $A$ has a basis, consisting of quasi-nilpotent elements,
the subalgebra $A$ is quasi-nilpotent. This proves $ (b) $.
\end{proof}
\begin{corollary}\label{cor1}
  A subalgebra $ A $ of the associative algebra $\End_\FF U $ is quasi-nilpotent
  if and only if the Lie algebra $ A_- $ is nilpotent and $ A $ has a basis, consisting of quasi-nilpotent endomorphisms. \qed
\end{corollary}

\section{Criterion of absolute indecomposability.}
Let $ \pi = (\{U_v\} , \{U_{v} \rightarrow U_{w}\}) $ be a representation of a quiver $ (\Gamma , \Omega) $ over a field $ \FF $, of dimension
$ \alpha = \sum_{v \in \mathcal{V}}  \ n_v \ \alpha_v $. Let $ U = \underset{v \in \mathcal{V}}{\oplus} U_v $. Then the space $ \operatorname{Hom}_\FF (U_{v}, U_{w}) $ is naturally identified with a subspace  of $ \End_\FF U $, so that the representation $ \pi $ is identified with a
collection of endomorphisms for each oriented edge $ v \rightarrow w $ of
the quiver $ (\Gamma , \Omega )$: $\{ \pi_{v,w} : U_{v} \rightarrow U_{w}   \}
\subset \End_\FF U$.
An endomorphism $ a $  of $ \pi $ decomposes as $ a = \sum_{v \in \mathcal{V}} a _v$, where $ a_v \in \End_\FF U_v \subset \End_\FF U $, and the condition that $ a \in \End \pi$, the algebra of endomorphisms of $ \pi $, means that 
\begin{equation}\label{5}
a_w \pi_{v, w} = \pi_{v,w} a_v \text{ for all oriented edges } v \rightarrow w.
\end{equation}
This simply means that the block diagonal endomorphism $ a $ commutes with all endomorphisms $ \pi_{v,w} $ in the algebra $ \End_\FF U. $ Note that (\ref{5}) has an obvious solution
$a_v=cI_{U_v}, v\in \V$, where $c\in \FF$, hence $\dim \End \pi \geq 1$. In the case of equality,
$\al$ lies in $\Delta_+$, and it is called a \emph{Schur vector}; in this and only in this case
a generic representation of dimension $\al$ is absolutely indecomposable \cite{K82}.

  \begin{lemma}\label{lem2}
    The representation $ \pi $ is absolutely indecomposable if and only if the algebra of its
    endomorphisms $ \End \pi $ is quasi-nilpotent in $ \End_\FF U. $
  \end{lemma}
\begin{proof}
  An endomorphism $ a  \in \End \pi \subset \End_\FF U \subset \End_{\bar{\FF}} \overline{U}  $ decomposes
  in a sum of commuting endomorphisms $ a = a_{(s)} + a_{(n)},  $ where the endomorphism
  $ a_{(s)} $ is diagonalizable and the endomorphone $ a_{(n)} $ is nilpotent (Jordan decomposition). Condition \eqref{5} means that $ a $ commutes with $ \pi_{v,w} $ for all
  oriented edges $ v \rightarrow w. $ By a well-known fact of linear algebra, it follows that the $ \pi_{v,w} $ commute with $ a_{(s)} $. But then the decomposition of $\overline{ U} $ in a direct sum of eigenspaces of $ a_{(s)} $ is a decomposition of the representation $ \pi $ in a direct sum of representation of the quiver $ (\Gamma, \Omega)$. Thus, $\pi $ is absolutely
  indecomposable if and only if $ a_{(s)} $ is a scalar endomorphism of $\overline{U}, $ which is equivalent to say that $ a $ is a quasi-nilpotent endomorphism of $ U. $
\end{proof}

\section{Main Theorem}
The following is the main result of the paper.
\begin{theorem}\label{th2}
  Let $\FF_q$ be a fixed finite field. Then there exists an algorithm which, given as input a
  quiver  $(\Gamma, \Omega)$ and its representation $ \pi = (\{U_v\}, \{U_{v} \rightarrow U_{w} \} )$ over $\FF_q$ of dimension $ \sum_{v \in \mathcal{V}} n_v \alpha_v, $
  can decide in polynomial in $ N:= \sum_v n_v $ time whether $ \pi $ is absolutely indecomposable or not.
\end{theorem}
\begin{proof}
  By Lemma \ref{lem2} one has to check whether $ \End \pi \subset \End_{\FF_q} U, $ where
  $ U = \underset{v \in \V}{\bigoplus} U_v $, consists of quasi-nilpotent elements.
By Corollary \ref{cor1} one has to check two things :
\begin{enumerate}
	\item[(i)] $ \End \pi $ has a basis, consisting of quasi-nilpotent elements;
	\item[(ii)] the Lie algebra $ ( \End \pi)_- $ is nilpotent.
\end{enumerate}
For this we identify $ U_v $ with the vector space $ \FF^{n_v}_q, $ so that $ U $ is identified with $ \FF^N_q $ and $ \End_{\FF_q} U $ with the algebra of $ N \times N $- matrices over $ \FF_q. $ $ \End \pi $ is a subspace of $ \End_{\FF_q} U, $ given by linear homogeneous equation \eqref{5}, hence, using Gauss elimination, we can construct in polynomial in $ N $ time a basis $ a_1, \ldots, a_m $ of $ \epi, $ where $ m \leq N. $

First, we check that all the $ a_i $ are quasi-nilpotent. This simply means that
\begin{equation}\label{eq6}
{\det}_{U} (\la I_N + a_i) = (\lambda + \gamma_i)^N,\, \text{where} \, \gamma_i \in \overline{\FF}_q.  
\end{equation}
The left-hand side of (\ref{eq6}) can be computed in polynomial in $N$ time by Gauss elimination. By the separability of $\bar{\FF}_q$ over $\FF_q$, (\ref{eq6}) implies that all $\gamma_i$ lie in $\FF_q$. Hence we have to check that (\ref{eq6}) holds for each $i$ and some element $\gamma_i \in \FF_q$, which can be done in polynomial in $ N $ time.

Second, we check that $ (\epi)_- $ is a nilpotent Lie algebra.  Recall that a Lie algebra $ \fg $ of dimension $ m $ is nilpotent if and only if the member $ \fg^m $ of the sequence of subspaces, defined inductively by 
\[ \fg^1 = \fg, \quad \fg^j = [\fg , \fg^{j-1}] \text{ for } j \geq 2, \]
is zero. Given a basis $ \{ a_i \} $ of $ \fg $ (which we already have), the subspace $ \fg^2 $ is the span over $ \FF_q $ of all commutators $[a_i , a_j  ].  $ Using Gauss elimination, construct a basis $ \{ b_i   \} $ of $ \fg^2. $ Next, $ \fg^3 $ is the span of commutators $ [a_i , b_j   ] , $ and again, using Gauss elimination, choose a basis $ \{ c_i  \} $ of $ \fg^3, $ etc.  The Lie algebra $ \fg $ is nilpotent if and only if $ \fg^m = 0. $
\end{proof}

\section{A brief discussion on P vs NP}
In terms of matrices over $ \FF_q, $ a representation $\pi$ over $\FF_q$ of a quiver $ (\Gamma, \Omega) $ of dimension $ \alpha = \sum_{v \in \mathcal{V} } n_v \alpha_v \in Q_+ $ is a collection of $ n_{w} \times n_{v} $ matrices $ \pi_{v, w} $ over $ \FF_q $ for each oriented edge  
$ v \longrightarrow \hspace{-1ex} w  $. An endomorphism of $ \pi  $ is a collection of $ n_v \times n_v $ matrices $ a_v $ over $ \FF_q $ for each vertex $ v \in \V $, such that the linear homogeneous equations \eqref{5} hold. The representation $ \pi $ is absolutely indecomposable if for each endomorphism of $ \pi $ all matrices $ a_v, v \in \V $, are quasi-nilpotent (equivalently, by
Corollary \ref{cor1}, $ \epi $ has a basis of elements with this property).

The following discussion was outlined to me by Mike Sipser. Given a representation $ \pi $
over a fixed finite field $\FF_q$ of a quiver $ (\Gamma, \Omega) $ of dimension $ \al \in \Delta_+ $, which is a collection of
$ M_\alpha := \sum_{v \rightarrow w} n_v n_w$ numbers from $ \FF_q $, the output is YES if $ \pi $ is absolutely indecomposable and NO otherwise. Call this problem INDEC; it is a P problem, according to Theorem \ref{th2}. Define a generalization of INDEC, where some of the numbers are replaced by variables $ x_i, \ i=1, \ldots, M $, where $M$ is an integer, such that
$1\leq M \leq M_\al$,  and call this problem
INDEC$ [x_1, \ldots, x_{M}] $. Say YES for the latter problem if there exist $ \gamma_1, \ldots, \gamma_{M} \in \FF_q $ we can substitute for $ x_1, \ldots, x_{M} $, such that the resulting INDEC problem is YES. Obviously INDEC is in P implies that INDEC$ [x_1, \ldots, x_{M}] $ is in NP.

Now assume that INDEC$[x_1,...,x_{M_\al}]$ is actually in P.
We give a polynomial in $ M_\al $ time procedure to output an absolutely indecomposable representation. Test INDEC$ [x_1, \ldots, x_{M_\al}] $. The answer is YES by Theorem \ref{th1}(d).
Now reduce $ M_\al $ by 1, by trying all possible numbers from $ \FF_q $ in place of
$ x_{M_\al} $ and test INDEC$ [x_1, \ldots, x_{M_\al -1}] $ for each of these numbers. The answer must be YES for at least one of these numbers. Repeat this procedure until we find all $ M_\al $ numbers. That is our answer. 

\section{ Conjectures and Examples}

\begin{conjecture}\label{con1}
  INDEC$[x_1, \ldots, x_{M_\al}]$
  is not in P.
\end{conjecture}
\begin{conjecture}\label{con2}
 INDEC$[x_1, \ldots, x_{M_\al}]$ is in P for any quiver $ (\Gamma , \Omega   ) $ if $ \alpha \in \Delta_+ $ is a Schur
  vector.
\end{conjecture}
\begin{conjecture}\label{con3}
INDEC$[x_1, \ldots, x_{M_\al}]$ is in P for any quiver $(\Gamma , \Omega   )  $ if $ \alpha \in \C $ (defined by \eqref{eq4}).
\end{conjecture}

\begin{example}\label{ex1}
Let $ \Gamma  $ be a Dynkin diagram of type $ A_r, D_r, E_6, E_7, E_8. $  In this case for any orientation $ \Omega $ of $\Gamma$ all indecomposable representations have been constructed explicitly in \cite{G72}, which shows that in this case INDEC$[x_1,\ldots, x_{M_\al}]$ is in P.
\end{example}
\begin{example}\label{ex2}
Let $ \Gamma $ be the extended (connected) Dynkin diagram, so that $ \# \mathcal{V} = r + 1 $
and $ \det A = 0. $ These are the only connected graphs, for which the Cartan matrix is
positive semidefinite and singular. In this case all absolutely indecomposable
representations for any
orientation $ \Omega $ have been constructed in \cite{N73} and in \cite{DF73}.
which shows that in this case  INDEC$[x_1, \ldots, x_{M_\al}]$ is in P as well. Note that in this case \cite{K80}
$ \Delta^{\text{im}}_+ = \ZZ_{\geq 1} \delta, $ where $ A \delta = 0 $ and $ (\delta | \delta    ) = 0,  $ and one can show that  $ P_{\Gamma, n \delta} (q) = q + r$ for $n\in \ZZ_{\geq 1}$.
\end{example}
\begin{example}\label{ex3}
  Let $ \Gamma_m $ be the quiver with two vertices $ v_1 $ and $ v_2 $, and $ m $ arrows from $ v_1 $ to $ v_2. $ For $ m = 1 $ and 2 this is a quiver from Examples \ref{ex1} and \ref{ex2} respectively. For $ m \geq 3 $ the explicit expressions for the polynomials
  $ P_{\Gamma_m , \alpha}(q) $ for an arbitrary
  $\alpha \in \Delta^{\text{im}}_+$ are unknown. Note that in this case
  $ \Delta^{\text{re (resp. im)}}_+  = \{ \alpha = n_1 \alpha_1 + n_2 \alpha_2 |\, n_i \in \ZZ_{\geq 0} \text{ and }
  n^2_1 + n^2_2 - m n_1 n_2   = 1 \ ( \text{resp.}< 0 )\}.   $
\end{example}

Now, let $ (\Gamma , \Omega    ) $ be a quiver, and let $ v $ be a vertex, which is a source or a sink.  In \cite{BGP73} an explicitly computable reflection functor $ R_v $ was constructed, which sends a representation $ \pi $ of dimension $ \alpha \neq v $  of $ (\Gamma, \Omega) $ to a representation $ R_v(\pi) $ of the reflected quiver $ (\Gamma, R_v (\Omega)) $ of dimension $ r_v (\alpha),  $ preserving indecomposability, see also \cite{K80}. It follows that if the problem INDEC$[x_1, \ldots, x_{M_\al}]$ is in P for the quiver $ (\Gamma, \Omega) $ and dimension $ \alpha \neq v, $ and $ v  $ is a source or a sink of $ (\Gamma, \Omega) $, then it is in P for the quiver $ (\Gamma, R_v (\Omega)) $ and dimension $ r_v (\alpha) $.

\begin{remark}\label{rem1}
  If $ v $ is a source or a sink of the quiver $ (\Gamma , \Omega   ) $ and $ \alpha \in \Delta_+ \backslash \{v\} $ is a Schur vector, then $ r_v (\alpha) $ is a Schur vector for $ (\Gamma , R_v (\Omega)   ).$ Also, if $\alpha$ is such that INDEC$[x_1, \ldots, x_{M_\al}]$ is in P,
  then the same holds for $r_v(\al)$.
\end{remark}
\begin{remark}\label{rem2}
  For an arbitrary quiver $ (\Gamma , \Omega)  $ the set $ \mathcal{C} $ consists of Schur
  vectors, except for the vectors with $ ( \alpha | \alpha ) = 0 $ \cite{K80}, in which case, $ \operatorname{supp} \alpha $ is a graph from Example \ref{ex2}.  Hence Conjecture \ref{con2} implies Conjecture \ref{con3}.	
\end{remark}
\begin{remark}\label{rem3}
  Let $ \Gamma_m $ be a quiver from Example \ref{ex3}.  Then, using the reflection functors, we see that for all
  $ \alpha \in \Delta^{\text{re}}_+ $,
  INDEC$[x_1, \ldots, x_{M_\al}]$ is in P.  Since for this quiver
  $(\al|\al)<0$ for  all $\alpha\in \mathcal{C} $, we see that all $ \alpha \in \Delta^{\text{im}}_+ $ are Schur vectors \cite{K80}, and it follows from Remark \ref{rem1} and Conjecture \ref{con2} that for all $ \alpha \in \Delta^{\text{im}}_+ $, INDEC$[x_1, \ldots, x_{M_\al}]$ is in P as well. 
\end{remark}

However, in general, $ \alpha \in \Delta_+ $ is not a Schur vector, so that a generic representation of a quiver $ (\Gamma , \Omega) $ of dimension $ \alpha \in \Delta_+ $ is not absolutely indecomposable.
In this case INDEC$[x_1, \ldots, x_{M_\al}]$ becomes a problem of finding a needle in a haystack, which leads me to
(naively) believe in Conjecture \ref{con1}. 

In fact, I believe that for any connected quiver, different from those in Examples \ref{ex1}, \ref{ex2}, and \ref{ex3}, there exists $ \alpha \in \Delta_+ $, for which INDEC$[x_1, \ldots, x_{M_\al}]$ is not in P.

\begin{remark}\label{rem4}
  As explained in \cite{K83}, claim (a) of Theorem \ref{th1} extends to the case of $\Gamma$ with self-loops. Claim (c) is proved in \cite{HLRV13} in this generality.
Theorem \ref{th2} holds in this generality as well.  
\end{remark}
  
\section*{Acknowledgements}
I am grateful to L. Babai, L. Levin, S. Micali, B. Poonen, M. Sipser, M. Sudan, and R. Williams for very valuable discussions and correspondence. 

\end{document}